%% file: Logaritmo.tex
\documentclass[preprint,12pt]{elsarticle}
\usepackage{epsfig}
\usepackage{algorithmicx}
\usepackage[ruled]{algorithm}
\usepackage{algpseudocode}
\usepackage{amsmath}
\usepackage{amssymb}
\usepackage{subfig}
\usepackage{xspace}
\usepackage{tabularx}
\usepackage{url}
\usepackage{multirow}
\usepackage{amsthm}
\newtheorem{theorem}{\bf Theorem}[section]

\newproof{pf}{Proof}

\newtheorem{conjecture}[theorem]{\bf Conjecture}
\newtheorem{remark}[theorem]{\bf Remark}

\newcommand{\logmpol}{\texttt{logm\_pol}\xspace}
\newcommand{\logmpcg}{\texttt{logmpcg}\xspace}

\newcommand{\logmnew}{\texttt{logm\_new}\xspace}

\begin{document}

\begin{frontmatter}



\title{Polynomial approximations for the matrix logarithm with computation graphs}


\author[KTH]{E. Jarlebring}
\ead{eliasj@kth.se}
\author[ITEAM]{J. Sastre}
\ead{jsastrem@upv.es}
\author[I3M]{J. Javier}
\ead{jjibanez@upv.es}

\address[KTH]{Department of Mathematics, KTH Royal Institute of Technology in Stockholm,
Lindstedtsvägen 25, 114 28 Stockholm, (Sweden)}
\address[ITEAM]{Instituto de Telecomunicaciones y Aplicaciones Multimedia,
Universitat Polit\`ecnica de Val\`encia, Camino de Vera s/n, 46022-Valencia (Spain)}
\address[I3M]{Instituto Universitario de Matemática Multidisciplinar,
Universitat Polit\`ecnica de Val\`encia, Camino de Vera s/n, 46022-Valencia (Spain)}

\begin{abstract}
  The most popular method for computing the matrix logarithm is a combination of the inverse scaling and squaring method in conjunction with a Pad\'e approximation, sometimes accompanied by the Schur decomposition.
  The main computational effort lies in matrix-matrix multiplications and
  left matrix division. In this work we illustrate that the number of
  such operations can be substantially reduced, by using a graph based
  representation of an efficient polynomial evaluation scheme.
  A technique to analyze the rounding error is proposed,
  and backward error analysis is adapted. We provide substantial
  simulations illustrating competitiveness both in terms of computation
  time and rounding errors.
\end{abstract}


\end{frontmatter}

\input{Sec_intro}
\input{Sec_background}
\input{Sec_polynomial}
\input{Sec_CoeffsStability}
\input{Sec_bacward_stability}
\input{Sec_algorithm}
\input{Sec_tests}

\section{Conclusions} \label{conclusions}
In this paper, an algorithm to approximate the matrix logarithm has
been designed. It uses recent matrix polynomial evaluation formulas
\cite{sastre2018efficient, SaIb21, Jar2021} more efficient than the
matrix polynomial evaluation formulas given by the
Paterson--Stockmeyer method \cite{PaSt73}. The approximations are
based on combining Taylor based approximations
\cite{Iba2021Improved} and computation graphs \cite{Jar2021}
corresponding to min--max approximations of the logarithm function
in a complex disks. Moreover, an improved matrix scaling algorithm
based on a backward error analysis has been provided. A MATLAB code
has been implemented. In tests the new implementation offer a higher
accuracy than the state-of-the-art Pad\'e implementations, and
slightly lower accuracy than the recent efficient Taylor-based code
from \cite{Iba2021Improved}. It also gives a lower computational
cost than the other state-of-the-art MATLAB implementations with no
built-in functions nor BLAS optimizations. Future work will be
implementing code optimizations in the algorithm and increasing the
accuracy to the level of the previous polynomial implementation
\logmpol.

\section*{Acknowledgements}\label{Ack}
Work supported by the Vicerrectorado de Investigaci\'on de la Universitat Polit\'ecnica de Val\'encia (PAID-11-21 and PAID-11-22) and by the KTH Royal Institute of Technology in Stockholm. We want to thank the authors of \cite{fasi2020dual} Dr. Fasi and Dr. Iannazzo for providing the MATLAB code \texttt{logm\_ex} from \cite[Alg. 2]{fasi2020dual} for the numerical tests.

\bibliographystyle{elsarticle-num}
\bibliography{bib_general.bib}


\end{document}

%% file: Sec_intro.tex
\section{Introduction and notation}\label{section_introduction}
This paper concerns the matrix logarithm, e.g., defined as the inverse problem corresponding to the matrix exponential.
A logarithm of a square matrix $A \in \mathbb{C}^{n \times n}$ is any matrix $X \in \mathbb{C}^{n \times n}$ satisfying the matrix equation
\begin{equation}\label{log0}
A=e^{X}.
\end{equation}
From \cite[Th. 1.27]{High08}, if $A$ is nonsingular, then there are infinite solutions of equation (\ref{log0}). If $A$ has no eigenvalues in $\mathbb{R^-}$, $\log(A)$, called the principal logarithm of $A$, is the only solution of \eqref{log0} whose spectrum lies in the strip $\{z\in \mathbb{C}, -\pi<\operatorname{Im}(z)<\pi\}$. The principal logarithm will be denoted $\log(A)$.


The principal logarithm is a versatile tool widely used across diverse fields ranging from pure science to engineering, and we list a few applications, following \cite{miyajima2019verified}. Its applications extend to quantum chemistry and mechanics, e.g., mentioned in \cite{hesselmann2010random,zachos2007classical}. It is also used in buckling simulation \cite{schenk2009modeling} and biomolecular dynamics \cite{horenko2008likelihood}. In the realm of machine learning, its significance is highlighted in works like \cite{han2015large,fitzsimons2017entropic}. It can be used in graph theory, as evidenced by \cite{williams1999matrix,grindrod2014dynamical}, and it is instrumental in the study of Markov chains \cite{israel2001finding}, but also sociology \cite{singer1976representation}, optics \cite{ossikovski2015differential}, and mechanics \cite{ramezani2015non}. In the field of computer graphics, the principal logarithm finds its use, as demonstrated in \cite{rossignac2011steady}, and it is equally important in control theory \cite{lastman1991infinite} and computer-aided design (CAD) \cite{crouch1999casteljau}.



The method presented in this paper belongs to a class of methods
whose main computational effort lies in the computation of matrix-matrix products, and left matrix division. The basis of the approach is the
inverse scaling and squaring method \cite{KeLa89_2},
which we describe further in Section~\ref{sec:methods}.
The standard version of inverse scaling and squaring is based on a Padé approximant, i.e., a rational approximation. Although rational approximations have been very successful in this setting, it has been established that polynomials can also be very competetive for this method class, but only if particular efficient evaluation scheme are used that to keep the number of matrix-matrix multiplications low, e.g., with the algebraic determination of evaluation coefficients in \cite{sastre2018efficient} and optimization-based approach in \cite{Jar2021}; see Section~\ref{sec:methods} for further references. The main objective of this paper is to show that combining the works of \cite{Iba2021Improved} and \cite{Jar2021}, the optimization approach for efficient polynomial evaluation techniques,
even better algorithms can be obtained for the matrix logarithm.

More precisely, the main contributions of the paper are as follows.
A new technique to obtain a polynomial approximation of $\log(I-A)$,
which has a low number of matrix-matrix products, is given in
Section~\ref{Sec_polynomial}. It is illustrated that, for the new
evaluation schemes, the stability (in the sense of forward rounding
errors) can be analyzed with the technique described in
Section~\ref{Sec_CoeffsStability}, which is important for the choice
of essentially which evaluation scheme is to be used. The backward
error analysis for the standard inverse scaling and squaring is
adapted for the approximations in this setting is provided in
Section~\ref{Sec_back_analysis}.

In addition to the sections described above, we explain the details of the
algorithm in Section~\ref{Sec_algorithm} and illustrate the competitiveness
of the method with extensive numerical simulations in Section~\ref{Sec_experiments}.

We refer to the simulations section for thorough reporting of
computational results, but here provide an initial illustration
that our method can be competitive in terms of efficiency and error
 in comparison to partial fraction
evaluation approach (in \verb#logm_pade_approx#) used in the builtin MATLAB command \verb#logm#\footnote{Computing a Schur decomposition for this matrix took 20  seconds, and is therefore not used in this illustration.}:
\begin{verbatim}
>> randn('seed',0);
>> n=4096; A=randn(n,n); A=0.5*A/norm(A);
>> tic; P1=logm_pade_approx(A,5); toc  % Builtin
Elapsed time is 7.738816 seconds.
>> tic; P2=-logm_21p_opt(-A); toc;     % Proposed approach
Elapsed time is 5.127813 seconds.
>> norm(expm(P1)-(eye(n)+A)) % Error builtin
ans =
   6.3014e-12
>> norm(expm(P2)-(eye(n)+A)) % Error proposed approach
ans =
   6.0170e-15
\end{verbatim}

%% file: Sec_background.tex
\section{Related methods}\label{sec:methods}

\subsection{Methods for the matrix logarithm}\label{sec:log_methods}
Without an attempt to provide a thorough overview of the methods for the matrix logarithm, we will now discuss some of the methods most relevant for this work.

One of the most used methods for computing the matrix logarithm is
the inverse scaling and squaring method, which was originally
proposed in the context of condition number estimation by Kenney and
Laub in \cite{KeLa89_2}. Subsequently, the algorithm has received
considerable attention and been improved and further analyzed by
several authors,
\cite{KeLa89_1,dieci1996computational,High01,ChHK01,al2012improved,al2013computing,fasi2018multiprecision,fasi2020dual}.

The derivation of the inverse scaling and squaring method is based
on the identity
\begin{equation}\label{matrix_logarithm_identity}
    \log{(A)}=2^s \log{(A^{1/2^s})}.
\end{equation}
The approach (in the improved version \cite{al2012improved})
can be viewed as two steps
\begin{itemize}
\item[1.] Approximate $A^{1/2^s}$, by repeated application of a version of the Denman–Beavers (DB) iteration (see \cite[eq. (6.29)]{High08}), until  $A^{1/2^s}$ is close to the identity matrix.
\item[2.] Approximate $\log \left( I+B \right)$, with $B=A^{1/2^s}-I$ using a $[m/m]$ Pad\'e approximant $r_m(x)$ to $\log(1 + x)$ given by \cite{High01}.
\end{itemize}
Finally the approximation of  $\log(A)$ is recovered by computing ${2^s}{r_{m}}(A)$.

The inverse scaling and squaring method is not the only approach to the matrix
logarithm. A Schur-Fréchet algorithm was proposed in \cite{kenney1998schur}, involving a transformation to upper triangular form, and then using the Fréchet derivative of the function to evaluate the off-diagonals. Iterations schemes based on arithmetic-geometric mean for the scalar logarithm, can be generalized in a matrix function sense \cite{cardoso2016matrix}.  The Cauchy integral definition of
matrix functions, can be combined with quadrature to obtain competitive methods
for the logarithm  \cite{hale2008computing}; see also \cite{dieci1996computational,tatsuoka2020algorithms}. 

\subsection{Efficient polynomial evaluation}

The inverse scaling and squaring approach, which is the basis of this paper, belongs to a class of methods that only involve: (a) additions and scaling of matrices, (b) matrix-matrix multiplications and left matrix division. The operations (a) require computational effort $\mathcal{O}(n^2)$,  and (b) require $\mathcal{O}(n^3)$ operations, and therefore, when we discuss efficiency, operations of the type (a) will be neglected. The two types of operations in (b) can require substantially different computational
effort in practice on modern computers. For practical purposes, we use the canonical way to relate the computational effort for the two types of operations based on the LAPACK documentation.
The cost of evaluation one matrix product is denoted by $M$, and the cost left inverse multiplication is denoted by $D$.
From \cite[App. C]{BlDo99} and \cite[Equation (1)] {SASTRE2018229}, it follows that
\begin{equation}\label{equivDM}
 D \approx \frac{4}{3} M.
 \end{equation} Hence, although our approach
is polynomial based, we can, with this relation compare it with
rational approaches in terms of computation time.

The topic of evaluating polynomials with as few multiplications has
received substantial attention over the last decades. The approach
in  \cite{PaSt73}, called the Paterson--Stockmeyer approach,
represents an  advancement which is still very competitive today.
It shows how one can evaluate polynomials of degree $m$
with $\mathcal{O}(\sqrt{m})$ operations, which is the optimal in
terms of complexity order. See also the refinement and optimality discussion in
\cite{FASI2019182}.

Despite the fact that the Paterson--Stockmeyer method is (in a certain sense)
optimal, there are substantial recent improvements, for specific number of multiplications. In a sequence of works, by some authors of this paper, it has
been illustrated how one can form evaluation schemes that are more efficient.
For example, for three multiplications evaluation formulas
from~\cite[Sec. 2.1]{Iba2021Improved} can be used:
{\setlength\arraycolsep{2pt}
    \begin{eqnarray}\label{eqP8pol}
        A_2&=&A^2, \nonumber\\
        y_{02}(A)&=&A_2(c_4 A_2+c_3 A),\\
        y_{12}(A)&=&((y_{02}(A)+d_2 A_2+d_1 A)(y_{02}(A)+e_2 A_2)+e_0y_{02}(A)\nonumber\\
        &&+f_2A_2/2+f_1A+f_0I). \nonumber
\end{eqnarray}}
Note that this involves three multiplications, and it has 9 free
variables, the same number of variables that parameterize
polynomials of degree 8. In \cite[Sec. 3.1]{SASTRE2018229} the free
variables are determined to evaluate polynomials of degree $2^3=8$,
whereas Paterson-Stockmeyer for polynomials of degree 8 requires 4
multiplications \cite{FASI2019182}. The coefficients are determined
by solving a set of complicated algebraic equations, analytically
(by hand or with symbolic software). The equations do not have
solutions for the particular cases when  polynomials of degree $m=7$
are considered, i.e.,  polynomials of degree $m=8$ with zero leading
coefficients. Beside this case, the scheme can evaluate arbitrary
polynomials of degree $m=8$ with three multiplications. For more
multiplications the situation is different. For four
multiplications, we obtain polynomials of degree $m=16$. However, it
is believed that one cannot evaluate all polynomials of degree
$m=16$. In \cite[Section 5.1]{sastre2018efficient} and
\cite{Iba2021Improved} it is illustrated that the Taylor expansion
of degree 16, of several types of functions can almost be obtained
with four multiplications, in the sense that all coefficients but
one are matched, see generalization of this result in \cite[Section
3.1]{SaIb21}. This approximation of the Taylor expansion is denoted
$m=15+$. Moreover, the function approximating the Taylor expansion
is denoted $S_{m+}^T$. The same notation is used throughout this
paper.

In \cite{Jar2021}, the problem of polynomial evaluation was approached
using a representation of the function as a graph. Instead of algebraic equations, an optimization problem is solved. This is further described in Section~\ref{Sec_polynomial}.

%% file: Sec_polynomial.tex
\section{Polynomial approximations of the matrix logarithm} \label{Sec_polynomial}

Step 2 in the inverse scaling and squaring method described in Section~\ref{sec:log_methods}, requires a rational or polynomial approximation of $\log(I+A)$.
We propose to use a combination of the Taylor based approach in
\cite{Iba2021Improved} and the polynomial min-max optimization approach in \cite{Jar2021}.

We separate the choice based on the number of matrix-matrix product $k$.
\begin{itemize}
  \item For $k=1, \ldots, 4$ we use the Taylor approximants of degree correspondingly 2,4,8,14+  following the procedure in \cite[Section 2]{Iba2021Improved}.
  \item For $k>5$ the min--max polynomial approximations based on computational graphs from \cite{Jar2021} will be used, leading to polynomials of degree $2^k$.
\end{itemize}


Both of the cases above are based on a slightly transformed logarithm expression.
Rather than directly approximating the function $\log(I+A)$, we will carry out polynomial approximations for the negative function and function argument, namely
\begin{equation}\label{Taylor_log}
  f(A):=-\log (I - A) = \sum\limits_{i = 1}^\infty  \frac{1}{i} {A^i}\approx
  \sum\limits_{i = 1}^m  \frac{1}{i} {A^i}=:T_m(A)
\end{equation}
which has a convergent sum if $\rho (A)<1$ (see \cite[pp. 273]{High08}).
Note that in contrast to the Taylor expansion of $\log(I+A)$,
the Taylor series in \eqref{Taylor_log} has positive coefficients, which we observed to lead to less rounding errors, both in the construction of approximations and in the evaluations.  Approximations of $\log(I+A)$ can be recovered by reverting the transformation
\begin{equation}\label{logipASTM}
\log(I+A)\approx - T_m(-A).
\end{equation}

The above transformation, combined with the evaluation scheme similar to \eqref{eqP8pol}, i.e., \cite{sastre2018efficient}, formed the basis of the approach in
\cite{sastre2019boosting}, using Taylor series coefficients, which we also use for $k=1,\ldots,4$.

%
Main ingredients of \cite{Jar2021}, that are exploited here for $k>4$,
can be summarized as follows. Any matrix function evaluation procedure, can be viewed as computation graph where
every node is a computed matrix and the edges correspond to
how nodes are computed. The graph is a directed acyclic graph.
In the algorithms that can be considered in the framework of \cite{Jar2021}, there are three types of operations:
\begin{itemize}
    \item[(a)] Forming a linear combination of two matrices $C=a A+ bB$ where $a,b\in\mathbb{C}$
    \item[(b)] Multiplying two matrices $C=AB$
\end{itemize}
In \cite{Jar2021} left division is also considered, but it is not needed here since this work only considers polynomial approximations.

The advantage of using computation graphs, is the fact that we can now view the coefficients in (a), i.e., the linear combination coefficients, as free variables and try to fit them such that we approximate a given target function well. Moreover, the graph representation allows for differentiation with respect to the scalar coefficients, by traversing the graph backwards. The use of derivatives allows for better function fitting routines, by enabling the use of optimization methods that depend on derivatives.

Our target funtion is $f$ in the left-hand side in \eqref{Taylor_log}, but for diagonalizable $A$ it can be reduced to a scalar logarithm with the inequality
\begin{equation}\label{eq:err}
 \|z(A)-f(A)\| \le \max_{z\in\Omega} |h(z)-f(z)|  \|V\|\|V^{-1}\|.
\end{equation}
Hence, we obtain a multivariate minimization problem
\begin{equation}\label{eq:minmax}
 \min_{x} \max_{z\in\Omega} |z(z)-f(z)|,
\end{equation}
where $\Omega$ is a disk containing the eigenvalues of $A$,
and $A=V\Lambda V^{-1}$ is a diagonalization of $A$ and $x$ a vector containing all the parameters in the graph,
i.e., all the parameters in the linear combinations.
Note that $z$ is given as a computation graph, for
given linear combination coefficients.
Essentially any method for  unconstrained optimization could be invoked to attempt to determine
the minima. Among many methods in \cite{Jar2021},
a Gauss-Newton method with regularization has given the best results.

\begin{remark}[Eigenvalue condition number]
The optimization problem \eqref{eq:minmax} arises from neglecting the  eigenvalue condition number $\|V\|\|V^{-1}\|$ in \eqref{eq:err}. This term may be large or even infinite for non-diagonalizable matrices. However, in our setting, this is only used in the optimization stage, and for any computed minimizer we carry out a backward error analysis, see Section \ref{Sec_back_analysis}. It may happen in theory that the optimizer finds a minimizer which does not work well for non-normal matrices. We have never observed this in practice, and our backward error analysis provides a guarantee for an error also for non-diagonalizable matrices.
\end{remark}
In addition to the computation graph a technique specifically
for polynomials is proposed in \cite{Jar2021}.
There is a general form for polynomial evaluation, called referred to as the degree optimal form.
Let $P_1=I$ and $P_2=A$, and
\begin{subequations}\label{eq:five_mult_degopt}
\begin{align}
  P_3&=(h_{1,1}P_1+h_{1,2}P_2)(g_{1,1}P_1+g_{1,2}P_2),\\
   P_4&=(h_{2,1}P_1+h_{2,2}P_2+h_{2,3}P_3)(g_{2,1}P_1+g_{2,2}P_2+g_{2,3}P_3)\\
   &\phantom{x}\vdots\nonumber\\
   P_{k+2}&=(h_{k,1}P_1+\cdots+h_{k,k-1}P_{k+1})
   (g_{k,1}P_1+\cdots+g_{k,k-1}P_{k+1}).
\end{align}
The output of the evaluation is formed from a linear combination of
the computed matrices
\begin{equation}
z(A)=y_1P_1+\cdots+y_{k+2}P_{k+2}.
\end{equation}
\end{subequations}
This parameterizes an evaluation scheme via the scalar coefficients,
in this notation $h_{1,1},\ldots,h_{5,7}$, $g_{1,1},\ldots,g_{5,7}$ and $y_1,\ldots,y_7$, and can be represented as a computation graph.

Although this parameterizes (in exact arithmetic) all polynomial
evaluation schemes, it is not a minimal parametrization. The
evaluation scheme \eqref{eq:five_mult_degopt} has $k^2+4k+2$ free
coefficients, and we illustrate in the following theorem and assume
in a conjecture that some coefficients can be eliminated without
loss of generality. The degrees of freedom is reduced to $k^2+4k-2$
and $k^2+2k-2$ respectively.

\begin{theorem}
  The set of polynomials parameterized by \eqref{eq:five_mult_degopt}
  is unchanged by the restriction
\begin{subequations}\label{eq:normalize_first}
  \begin{align}
    h_{1,2}&=g_{1,2}=1\label{eq:normalize_first_a}\\
    h_{1,1}&=g_{1,1}=0.\label{eq:normalize_first_b}
  \end{align}
\end{subequations}
\end{theorem}
\begin{proof}
  Given an evaluation scheme, we prove that there is a different evaluation scheme, represented by $\tilde{h}_{i,j}$, $\tilde{g}_{i,j}$ and $\tilde{y}_{i}$
  satisfying \eqref{eq:normalize_first} but resulting in the same output $z$. First, we may assume that $h_{1,2}$ and $g_{1,2}$ are non-zero, since otherwise $P_3$ is a polynomial of degree one, i.e., the first multiplication can be expressed without any matrix-matrix multiplications. We prove \eqref{eq:normalize_first_a} first. Let $\tilde{P}_3=\frac{1}{h_{1,2}g_{1,2}}P_3=(\tilde{h}_{1,1}P_1+P_2)(\tilde{g}_{1,1}P_1+P_2)$, i.e., scaling of the first row such that it satisfies \eqref{eq:normalize_first_a}. Now define,
  $\tilde{h}_{j,3}=h_{1,2}g_{1,2}h_{j,3}$ and  $\tilde{g}_{j,3}=h_{1,2}g_{1,2}g_{j,3}$ for $j=2,\ldots,k+2$,
  and we obtain that $P_4=\tilde{P}_4,\ldots,P_{k+2}=\tilde{P}_{k+2}$,
  i.e., the other $P$-matrices remain unchanged. The same transformation can be done for the $y$-vector, such that the output is unchanged.

  To prove \eqref{eq:normalize_first_b} we assume \eqref{eq:normalize_first_a}
  and expand $P_3=(h_{1,1}P_1+P2)(g_{1,1}P_1+P2)=h_{1,1}g_{1,1}I+ (h_{1,1}+g_{1,1})A+ A^2=\alpha I+\beta A+A^2$.  Let $\tilde{P}_3=P_2^2=A^2$. Hence, the difference between $P_3$ and $\tilde{P}_3$ consists of linear combination of $I$ and $A$. In every occuranceof $P_3$, we can now incorporate these terms by modifying the first columns of $h$ and $g$. More precisely, $P_j=\tilde{P}_j$, for $j=3,\ldots,k+2$ if we
  set
  $\tilde{h}_{j,1}=h_{j,1}+\alpha h_{j,3}$,
  $\tilde{g}_{j,1}=g_{j,1}+\alpha g_{j,3}$,
  $\tilde{h}_{j,2}=h_{j,2}+\beta h_{j,3}$ and
  $\tilde{g}_{j,2}=g_{j,2}+\beta g_{j,3}$ for $j=1,\ldots,k+1$. The same transformation can be done for the $y$-vector, implying that $z(A)=\tilde{z}(A)$.

\end{proof}
A key ingredient to the proof was that the matrix $P_3$ is
only changed with addition
of one term involving $I$ and one term involving $A$, when the condition
is imposed. If we set the constant coefficient corresponding to $P_4$ to
zero we make a modification involving a linear combination of
$I,A$ and $P_3=A^2$, justifying the following conjecture, stating that the constant coefficients in $h_{i,j}$ and $g_{i,j}$ can be set to zero.

\begin{conjecture}
  The set of polynomials parameterized by \eqref{eq:five_mult_degopt}
  is unchanged by the restriction
\[
  h_{j,1}=g_{j,1}=0,\;\;j=1,\ldots,k+1.
  \]
\end{conjecture}


The reduction of the degrees of freedom is important for several reasons.
First, it gives an indication of how large polynomial sets one can parameterize with a fixed number of multiplications. For example, for $k=6$ we have $58$ (or $46$) degrees of freedom, whereas it results in a polynomial of degree $2^k=64$.
Hence, although $z(A)$ in \eqref{eq:five_mult_degopt} is a polynomial of degree $2^k$,  we cannot expect that all polynomials of degree $2^k$ can be computed for $k\ge 6$. Moreover, and more importantly for our work, a reduction of the degrees of freedom simplifies the optimization problem \eqref{eq:minmax}. Less parameters to determine gives faster computation per iteration, and it also improves the conditioning, in practice giving much better solutions.

\begin{remark}[Using the conjecture in the optimization]
In our approach here, we will parameterize according to both the theorem and the conjecture. The fact that we have not proven the conjecture does not jeopordize the final result of this paper. If, contrary to what we expect, the conjecture is not true,
it simply means that there might be slightly better approximations than the ones found in this paper. We carried out simulations without the restriction and found no indication of a better minimizers without the conjecture conditions.
\end{remark}

The optimization routine is carried out for various number of multiplications, as we report in Section~\ref{Sec_experiments}. For reproducibility we also provide the result of the optimization for $k=5$ in Table~\ref{table_simple_experiment} with the coefficients in IEEE double precision arithmetic. The error of the approximation is illustrated in Figure~\ref{fig_simple_experiment}. A feature of our approach to work with the minmax problem, results in an error which is in the order of magnitude of machine precision, in a large domain. In contrast to both the Padé approach (default in MATLAB) or a Paterson-Stockmayer evaluation, which are moment matching, meaning (unneccessarily) high accuracy at the origin.  In this sense, our approach gives a sufficiently accurate approximation in a larger domain, with the same computation effort, in comparison to other approaches.

The rounding error analysis that is described in detail in
Section~\ref{Sec_CoeffsStability}, specializes for the values in
Table~\ref{table_simple_experiment} as follows. Substituting the
values of Table~\ref{table_simple_experiment} into
\eqref{eq:five_mult_degopt} and writing $z$ as a polynomial
$z=\sum_{i=0}^{32} b_i A^i$, using high-precision arithmetic, one
gets that the relative error with respect to the corresponding
Taylor coefficients results $|b_i-1/i|i\leq u$, $1\leq i\leq 14$,
where $u$ is the unit roundoff in IEEE double precision arithmetic
$u\approx1.1\times 10^{-16}$. This is the same order of magnitude as
approximation $m=14+$.


\begin{figure}[H]
        \centering
        \includegraphics[width=8cm]{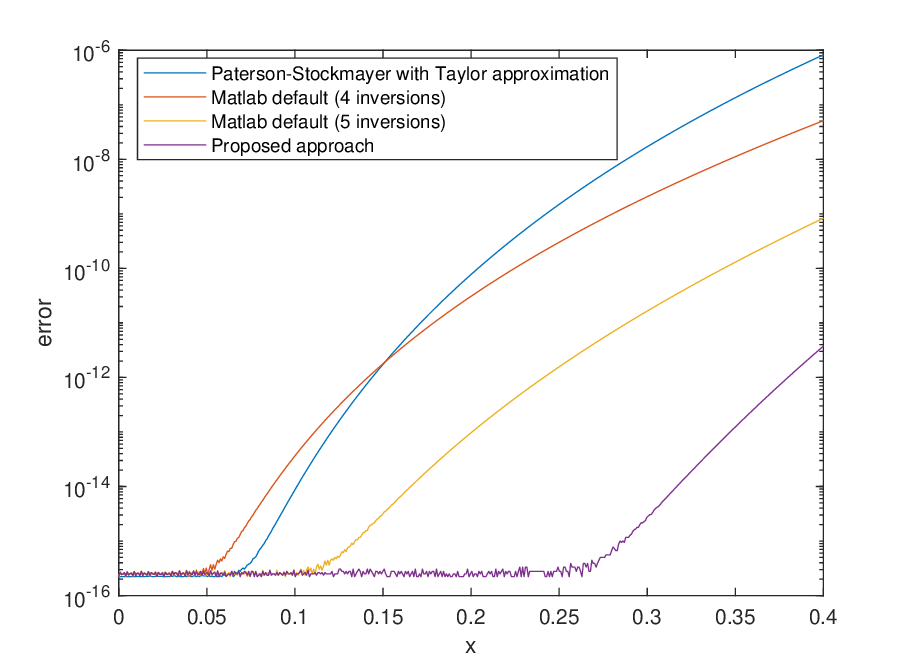}
    \caption{Comparison of the degree optimal polynomial obtain using the optimization scheme \eqref{eq:five_mult_degopt}. For comparison, also the Taylor approximation obtained with a Paterson--Stockmeyer evaluation with the same number of matrix-matrix multiplications is also provided.}
    \label{fig_simple_experiment}
\end{figure}

\begin{table}
        \begin{tabular}{| c || l | }
$h_{2,2}$ &     7.363757032799957e-02    \\
$h_{2,3}$ &    -1.050281301619960e+00    \\
$h_{3,2}$ &     8.897468955192446e-02    \\
$h_{3,3}$ &    -1.599651928992725e-01    \\
$h_{3,4}$ &     9.577281350989334e-01    \\
$h_{4,2}$ &     5.394999133948797e-01    \\
$h_{4,3}$ &     6.700731102561937e-02    \\
$h_{4,4}$ &    -5.158769100223212e-02    \\
$h_{4,5}$ &     1.094308587350110e+00    \\
$h_{5,2}$ &     1.027072285939197e-01    \\
$h_{5,3}$ &    -8.964023050065877e-03    \\
$h_{5,4}$ &    -2.100705663612491e-01    \\
$h_{5,5}$ &     1.949655359168707e-01    \\
$h_{5,6}$ &     1.117368056772713e+00    \\
\end{tabular}%
        \begin{tabular}{| c || l | }
$g_{2,2}$ &    -9.666134174379001e-01    \\
$g_{2,3}$ &    -4.395519034717933e-01    \\

$g_{3,2}$ &     1.048664069004776e-01    \\
$g_{3,3}$ &     1.585606124033259e-01    \\
$g_{3,4}$ &     1.668066506920988e-01    \\
$g_{4,2}$ &    -8.025600931705978e-02    \\
$g_{4,3}$ &    -1.159854366397558e-01    \\
$g_{4,4}$ &     1.066554944706011e-01    \\
$g_{4,5}$ &     1.127094008297975e+00    \\
$g_{5,2}$ &     2.702180425508705e-01    \\
$g_{5,3}$ &     4.137541209720699e-02    \\
$g_{5,4}$ &     4.857347452405025e-01    \\
$g_{5,5}$ &    -6.000256005636980e-01    \\
$g_{5,6}$ &     1.063393233943084e+00    \\
\end{tabular}
                \begin{center}
          \begin{tabular}{| c || l | }
$y_1$ &     0.000000000000000e+00    \\
$y_2$ &     1.000000000000000e+00    \\
$y_3$ &     5.065546620208965e-01    \\
$y_4$ &     3.832512052972577e-01    \\
$y_5$ &     1.088307723749078e+00    \\
$y_6$ &     2.787461897212877e-01    \\
$y_7$ &     8.157421998489228e-01
        \end{tabular}
\end{center}
          \caption{Coefficients for the evaluation scheme \eqref{eq:five_mult_degopt} optimized for the function $-\log(1-x)$.}
        \label{table_simple_experiment}
\end{table}

%% file: Sec_CoeffsStability.tex
\section{Stability of the polynomial evaluation methods}\label{Sec_CoeffsStability}
The approaches described in the previous section, in general give
several possible evaluation schemes for a given number of
matrix-matrix multiplications, due e.g., to the fact that we
may have several local minima, or several essentially equivalent minima,
or that a set of algebraic equations may have several solutions.
In practice, we have observed (in this work and related works)
that the choice of which evaluation scheme is used matters considerably
in practice, when we evaluate the evaluation scheme
for matrices, as a consequence of rounding errors. Analysis of the
floating point arithmetic, e.g., using the standard model \cite[Section~2.2]{High02}
is not feasible and gave drastic overestimates of the error.
Instead we propose here to use a heuristic that helps us
select which evaluation scheme is better.

The selection is based on a comparison of monomial coefficients,
when the input coefficients are rounded to double precision, as
follows. Let $p(c;x)$ be the  polynomial resulting from applying the
evaluation scheme with the coefficients $c\in\mathbb{R}^N$. For
approximations of orders $2,4,8,14+$ we use the coefficients from
\cite[Section 2]{Iba2021Improved}, and for the min--max
approximations we use the coefficients $x_a,\ldots,$. Associated
with this polynomial, we let $b(c)$ be the monomial coefficients,
i.e.,
\[
p(c;x)=b_0(c)+b_1(c)x+\cdots b_m(c)x^m.
\]
If we let superscript denote differentiation with respect to $x$, we have
that
\begin{equation}\label{eq:aj}
b_j(c)= \frac{1}{j!}p^{(j)}(c;0).
\end{equation}
Now let $\tilde{b}(c)$ be the same function but all operations
subject to rounding errors,
and let $\tilde{c}$ be the coefficients in double precision.

We propose to use the relative error
\begin{equation}\label{eq:stability_error}
|b(c)-\tilde{b}(\tilde{c})|/|b(c)|,
\end{equation}
where $|\cdot|$ and $/$ are interpreted elementwise.
Due to the identity \eqref{eq:aj} this quantifies the results of
rounding error for the $j$th derivative,
evaluated at the origin.

In order to compute the indicator \eqref{eq:stability_error}, the reference vector $b(c)$ must be computed. For the Taylor based part (four or less multiplications, cf Section~\ref{Sec_polynomial}) we can analytically compute the exact values from the Taylor series, as was done similarly in \cite{Iba2021Improved}.  For the computational graph evaluation (five or more multiplications, cf Section~\ref{Sec_polynomial})   the min--max approximations $b(c)$ was computed by using high-precision arithmetic (300 decimal digits).




The best   candidate evaluation scheme (in the sense of \eqref{eq:stability_error}) were the following.
{\setlength\arraycolsep{2pt}
\begin{eqnarray}\label{stabcompgraph}
        \max\{|(b_i-\tilde{b}_i)/b_i|\}&\le& 5.20u=O(u), \text{approximation degree } 2^5=32,\\
        \max\{|(b_i-\tilde{b}_i)/b_i|\}&\le& 6.03u=O(u),\text{approximation degree } 2^6=64,
\end{eqnarray}}
For higher degrees $128,\,256$ and $512$ errors are up to $O(100u)$, but this noticeable errors occur in coefficients $b_i$ with large $i$, corresponding to high powers of the matrix $A$, and then they should not have a noticeable impact on the overall rounding error since $\rho(A)<1$. 
{\setlength\arraycolsep{2pt}
    \begin{eqnarray}\label{stabcompgraph128_512}
        \max\{|(b_i-\tilde{b}_i)/b_i|\}&\le& 6u,\text{  for } i<107, \text{ approx. degree } 2^7=128,\\
        \max\{|(b_i-\tilde{b}_i)/b_i|\}&\le& 6u,\, \text{  for } i<35, \text{ approx. degree }  2^8=256,\\
        \max\{|(b_i-\tilde{b}_i)/b_i|\}&\le& 6u,\, \text{  for } i<143, \text{ approx. degree }  2^9=512,
\end{eqnarray}}where $\max\{|(b_i-\tilde{b}_i)/b_i|\}=14.38u, 41.06u$ and $107,80u$ respectively.

We also checked that the coefficients $\tilde{b}_i$ for $i>m_k$ from all the min--max computation graph evaluation formulas selected in
Section~\ref{Sec_polynomial} have a relative error with respect to the Taylor coefficients of $-\log(I-A)$
less than or equal to $1$, i.e.
\begin{equation}
|(1/i-\tilde{b}_i)i|\le 1,\ i>m_k,\ k=5,\,6,\,\ldots,9,
\end{equation}behaving similarly to the corresponding coefficients $\tilde{b}_i$ for $i>m_4=14$ from (20) \cite{Iba2021Improved}, and similarly to all the rest of approximations from \cite[p. 6]{Iba2021Improved}.

%% file: Sec_bacward_stability.tex
\section{Backward error analysis} \label{Sec_back_analysis}
The error analysis for all the polynomial approximations used in this paper follows the ideas in \cite[Sec. 2.3]{Iba2021Improved}. In this paper we will focus on the backward error analysis, since it provided a lower cost on \cite[Sec. 2.3]{Iba2021Improved}. Note that the backward error analysis also provided a lower cost for the Padé algorithms from  \cite{al2012improved}.

        The absolute backward error can be defined as the matrix $\Delta A$ that verifies
        \[\log \left( {I + A + \Delta A} \right) = {-S^T_m(-A)}.\]
  where $S^T_m(A)$ is any of the approximations appearing in Section~\ref{Sec_polynomial}, where $m$ is given by Table \ref{Table_theta} Then
\[ I + A + \Delta A = {e^{{-S^T_m(-A)}}},\]
        whereby the backward error can be expressed as
    \begin{equation}\label{backwarderror}
        \Delta A = {e^{{-S^T_m(-A)}}} - I - A.
    \end{equation}

        For the case of the approximations matching Taylor coefficients from
        Section \ref{Sec_polynomial} in exact arithmetic it can be proved that this error can be expressed as
\begin{equation}\label{DeltaAseries}
\Delta A=\sum\limits_{k \ge m+1} {{c_k}{A^k}}.
\end{equation}However, when using the corresponding evaluation formulas with their coefficients rounded to IEEE double precision arithmetic, expression \eqref{DeltaAseries} is an approximation. For instance, if the evaluation coefficients rounded to IEEE double precision arithmetic are used for the Taylor matching approximation of order $m_{14+}$ it follows that computing symbolically the backward error coefficients $c_k$ gives $c_0=c_1=0$ and $0<|c_k|\le 3.95 \times 10^{-17}=O(10^{-17})<u$ for $k=3,\,4,\ldots,\,14$.

Similary, for the case of the min--max computation graph evaluation formulas, using in \eqref{DeltaAseries} the evaluation formula coefficients rounded to the IEEE double precision in the approximation $m_5=14+$, corresponding to polynomial degree $2^5=32$, gives $c_0=c_1=0$ and $0<|c_k|<4.76 \times 10^{-17}=O(10^{-17})<u$ for $k=3,\,4,\ldots,\,14$. 
The same happens with the rest of min--max computation graph approximations, i.e. $c_0=c_1=0$ and $0<|c_k|\le O(10^{-17})<u$ for $k=3,\,4,\ldots,\,m_k$.

Hence, we will also use \eqref{DeltaAseries} for the case of the min--max computation graph evaluation formulas, since those formulas also match Taylor coefficients rounded to IEEE double precision arithmetic up to order $m_k$. Therefore, the relative backward error of approximating $\log(I+A)$ by means of ${-S^T_m(-A)}$, denoted by $E_{rb}$, can be approximately bounded for all the evaluation formulas from section \ref{Sec_polynomial} as follows:
{\setlength\arraycolsep{2pt}
        \begin{eqnarray}
            \label{Eb}
            {E_{rb}}(A) &\approx& \frac{{\left\| {\sum\limits_{k \ge m + 1} {{c_k}{A^k}} } \right\|}}{{\left\| A \right\|}}\\ &=& \left\| {\sum\limits_{k \ge m + 1} {{c_k}{A^{k - 1}}} } \right\| = \left\| {\sum\limits_{k \ge m} {{\hat{c}_k}{A^k}} } \right\| \le \sum\limits_{k \ge m} {\left| {{\hat{c}_k}} \right|\alpha _m^k}  \equiv {\bar h_m}({\alpha _m}),\nonumber
        \end{eqnarray}}where $\hat{c_k} = {c_{k + 1}}$ and
\begin{equation}\label{def_alpha}
\alpha_m= \max \left\{ {||{A^k}|{|^{1/k}}:\,\,k \geqslant m} \right\}.
\end{equation}
Let  $\Theta_{m}$  be
        \begin{equation}\label{theta_Eb}
            {{\Theta }_m} = \max \left\{ {\theta  \ge 0:{{\bar h}_m}(\theta ) \le u} \right\}.
        \end{equation}

\begin{table}[ht]
    \caption{Values $\Theta_{m_{k}}$, $1\le k \le 9$, corresponding to the backward error analysis from \eqref{theta_Eb}, evaluation cost in terms of matrix products $M$, order $m_k$ and degree $2^k$ of the corresponding polynomial approximations.\label{Table_theta}}
    \begin{center}
    \begin{tabular}{ccccc}
    \hline
    $k=$ cost$(M)$&$m_k$& degree & $\theta_k$ \\
    \hline
        1&$2$&$\rm{2}$&$\rm{1.83} \cdot {10^{ - 8}}$\\
        2&$4$&$\rm{4}$&$\rm{1.53} \cdot {10^{ - 4}}$\\
        3&$8$&$\rm{8}$&$\rm{1.33} \cdot {10^{ - 2}}$\\
        4&$14+$&$\rm{16}$&$\rm{9.31} \cdot {10^{ - 2}}$\\
        5&$14+$&$\rm{32}$&$\rm{2.46} \cdot {10^{ - 1}}$\\
        6&$20+$&$\rm{64}$&$\rm{3.72} \cdot {10^{ - 1}}$\\
        7&$24+$&$\rm{128}$&$\rm{5.86} \cdot {10^{ - 1}}$\\
        8&$27+$&$\rm{256}$&$\rm{6.69} \cdot {10^{ - 1}}$\\
        9&$28+$&$\rm{512}$&$\rm{7.28} \cdot {10^{ - 1}}$\\
      \hline
    \end{tabular}
    \end{center}
\end{table}

        The values ${\Theta}_{m_{k}}$ listed in Table \ref{Table_theta} were obtained from \eqref{theta_Eb} by using symbolic computations in a similar way as those from the algorithm backward error analysis from \cite[Sec. 2.3]{Iba2021Improved}. Moreover, Table~\ref{Table_theta} presents the evaluation cost of the corresponding polynomial approximations from section~\ref{Sec_polynomial} in terms of matrix products $M$, their order $m_k$ of approximation and polynomial degree $2^k$. Taking into account \eqref{Eb} and \eqref{theta_Eb}, if $\alpha_m  \le \Theta _m$, then
        \begin{equation}\label{Des_Eb}
            {E_{rb}}(A)\le {\bar h_{m}}({\alpha _m}) \le  {\bar h_{m}}({{ \Theta}_{m}})\le u,
        \end{equation}i.e. the relative backward error is lower than the unit roundoff in double precision arithmetic.

  Table \ref{Table_thetaIba2021Improved} reproduces \cite[Tab. 3]{Iba2021Improved} with the $\bar \Theta_{m_k}$ values for the Taylor based algorithm proposed there, the evaluation cost of the corresponding polynomial approximations, their order $m_k$ of approximation and maximum polynomial degree given by (27) from \cite[p. 6]{Iba2021Improved}, i.e. $m_k=6s+3+p$ and $8s+p$ where $s$ and $p$ are given by \cite[Tab. 2]{Iba2021Improved}. As said above, the polynomial approximations in \cite[Tab. 3]{Iba2021Improved} intend to match the maximum number of Taylor terms in the polynomial approximation, with a maximum polynomial degree given by $8s+p$. The algorithm proposed in this paper uses the same scheme for costs lower than $4M$, and the min--max scheme from Section~\ref{Sec_polynomial} for cost$\ge 5M$, where and effort was done to match the maximum number of Taylor terms rounded to IEEE double precision arithmetic from several of the iterative min--max solutions obtained. Moreover, the maximum polynomial degree was obtained by using a degree optimal form, i.e. $2^k$. Therefore, those polynomial approximations approximate fewer Taylor terms, however, since the polynomial approximations have more terms, the values ${\Theta}_{m_{k}}$ listed in Table \ref{Table_theta} are greater than those on Table \ref{Table_thetaIba2021Improved} for the same cost, thus giving in general a lower cost.

\begin{table}[ht]
\caption{Values  ${\bar \Theta}_{m_{k}}$, $5\le k \le 13$ from  \cite[Tab. 3]{Iba2021Improved}, corresponding to the backward error analysis from the algorithm in \cite[Sec 2.3]{Iba2021Improved}, evaluation cost in terms of matrix products $M$, order $m_k$ and maximum degree of the corresponding polynomial approximations. Note that the values for $k=2,\,3$ and $4$ are the same as the ones in table \ref{Table_theta} which are used in this paper, and they have not been included in this table.\label{Table_thetaIba2021Improved}}
    \begin{center}
    \begin{tabular}{ccccc}
    \hline
    $k=$cost$(M)$&$m_k$& Maximum degree & ${\bar \Theta}_{m_{k}}$\\
    \hline
        5&$21+$&$\rm{24}$&$\rm{2.11} \cdot {10^{ - 1}}$\\
        6&$27+$&$\rm{32}$&$\rm{2.96} \cdot {10^{ - 1}}$\\
        7&$33+$&$\rm{40}$&$\rm{3.73} \cdot {10^{ - 1}}$\\
        8&$39+$&$\rm{48}$&$\rm{4.28} \cdot {10^{ - 1}}$\\
        9&$45+$&$\rm{54}$&$\rm{4.90} \cdot {10^{ - 1}}$\\
        10&$52+$&$\rm{63}$&$\rm{5.41} \cdot {10^{ - 1}}$\\
        11&$59+$&$\rm{70}$&$\rm{5.82} \cdot {10^{ - 1}}$\\
        12&$67+$&$\rm{80}$&$\rm{6.18} \cdot {10^{ - 1}}$\\
        13&$75+$&$\rm{88}$&$\rm{6.54} \cdot {10^{ - 1}}$\\
      \hline
    \end{tabular}
    \end{center}
\end{table}

  With respect to the comparison with Pad\'e algorithms from \cite{al2012improved}, according to \cite[Sec. 2.3]{High01} and taking into account \eqref{equivDM} it follows that the cost of the evaluation of the partial fraction form of the $[k/k]$ Pad\'e approximant $r_k(x)$ to $\log(1 + x)$ from (1.1) \cite{al2012improved} at a matrix argument is $kD=k \times 4/3M$. Table~\ref{Table_thetaPade} presents the $\theta_k$ values from the Padé backward error analysis given by \cite[Sec. 2]{al2012improved}, see \cite[Tab. 2.1]{al2012improved}, together with the cost of evaluating each of the $r_k(x)$ approximants in partial fraction form as in the algorithms given by \cite{al2012improved}, and the order of the approximations. 

 Comparing Tables \ref{Table_theta} and \ref{Table_thetaPade} it follows that for $4\leq k\leq9$ in Table \ref{Table_theta} giving $4M\leq$ cost$\leq 9M$ for the corresponding polynomial approximations, respectively, one gets $\theta_k < \Theta_{m_{k}}$. Hence, for those values of $k$ it follows that the Padé approximants will usually need more square roots, and then the overall cost will be greater. Moreover, for Padé approximants with $9.33M\le$cost$\le20M$ the $\theta$ values for Padé approximants are also lower than $\Theta_{m_9}=7.28\times 10^{-1}$ from the min--max algorithm corresponding evaluation formula. Then, the respective overall costs of evaluating the Padé approximants are higher in all those cases than the cost of the corresponding min--max evaluation formula giving $m_9$ approximation order.

 Finally, evaluating $r_{16}(A)$ costs $12.33M$ more than evaluating the min--max approximation for $m_9$ and $\theta_{16}=7.49\times 10^{-1}$ is not much greater than $\Theta_{m_9}=7.28 \times 10^{-1}$. Therefore, the overall cost considering also the cost of the roots will be higher than the overall cost for computing the min--max approximation and its associated roots with a high probability.

\begin{table}[ht]
    \caption{Values  ${\theta}_{k}$, $1\le k \le 16$, corresponding to the backward error analysis from algorithms based on diagonal Pad\'e $r_k$ approximants from \cite{al2012improved}, evaluation cost in terms of matrix products $M$, order $m_k$ of the corresponding approximations. \label{Table_thetaPade}}
    \begin{center}
    \begin{tabular}{cccc}
    \hline
    $k$ from $r_k$& cost$(M)$&Approximation order $m_k$& ${\theta}_{k}$\\
    \hline\hline
        1&$1.33$&$2+$&$\rm{1.59} \cdot {10^{ - 5}}$\\
        2&$2.67$&$4+$&$\rm{2.31} \cdot {10^{ - 3}}$\\
        3&$4.00$&$6+$&$\rm{1.94} \cdot {10^{ - 2}}$\\
        4&$5.33$&$8+$&$\rm{6.21} \cdot {10^{ - 2}}$\\
        5&$6.67$&$10+$&$\rm{1.28} \cdot {10^{ - 1}}$\\
        6&$8.00$&$12+$&$\rm{2.06} \cdot {10^{ - 1}}$\\
        7&$9.33$&$14+$&$\rm{2.88} \cdot {10^{ - 1}}$\\
        8&$10.67$&$16+$&$\rm{3.67} \cdot {10^{ - 1}}$\\
        9&$12.00$&$18+$&$\rm{4.39} \cdot {10^{ - 1}}$\\
        10&$13.33$&$20+$&$\rm{5.03} \cdot {10^{ - 1}}$\\
        11&$14.67$&$22+$&$\rm{5.60} \cdot {10^{ - 1}}$\\
        12&$16.00$&$24+$&$\rm{6.09} \cdot {10^{ - 1}}$\\
        13&$17.3$3&$26+$&$\rm{6.52} \cdot {10^{ - 1}}$\\
        14&$18.67$&$28+$&$\rm{6.89} \cdot {10^{ - 1}}$\\
        15&$20.00$&$30+$&$\rm{7.21} \cdot {10^{ - 1}}$\\
        16&$21.33$&$32+$&$\rm{7.49} \cdot {10^{ - 1}}$\\
      \hline
    \end{tabular}
    \end{center}
\end{table}

 Summarizing, the polynomial approximations proposed in this paper have less cost than the Padé approximations from \cite{al2012improved}. And in general, they also have less cost than the Taylor based polynomial approximations from \cite{Iba2021Improved}.

%% file: Sec_algorithm.tex
\section{Algorithm for computing $\log(A)$} \label{Sec_algorithm}

Given that $\log(A)$ can be rewritten as $\log(I + A - I)$, Algorithm~\ref{Alg_log_Taylor} calculates $\log(A)$ utilizing an inverse scaling and squaring method with a polynomial approximation. This approach incorporates the polynomial estimations outlined in Section~\ref{Sec_polynomial}. In this context, $m_k$, where $k$ ranges from 1 to 9 as per Table \ref{Table_theta}, represents the degree of the corresponding Taylor series approximation, tailored to fit within the confines of IEEE double precision arithmetic. Additionally, this algorithm employs $\Theta_{k}$, derived from the backward error analysis mentioned in the same table. The MATLAB code for Algorithm~\ref{Alg_log_Taylor} is accessible online:

\url{http://personales.upv.es/~jorsasma/software/logmpcg.m}

\noindent and specifically the code for the example in the
introduction

\url{http://personales.upv.es/~jorsasma/software/logm_21p_opt.m}

\begin{algorithm}[!ht]
\caption{Given a matrix $A \in {\mathbb{C}^{n \times n}}$, a maximum
approximation order $m_k$, $k \in \left\{ {1,2,\cdots ,9} \right\}$,
given in Table \ref{Table_theta}, and a maximum number of iterations
$maxiter$, this algorithm computes  $L=\log(A)$ by a Taylor based
approximation of order lower than o equal to $m_k$.}
\label{Alg_log_Taylor}
\begin{algorithmic} [1]
\State Preprocessing of matrix $A$: $A\leftarrow T^{-1}AT$ \label{Step_Preproc}
\State $s=0$ \label{Step_start_invscaling}
\State $finish=||A-I||\le\Theta_{m_9}$
\While{\textbf{not} $finish$ and $s<maxiter$}
    \State $\alpha_{m_9}=\max\{||(A-I)^{m_9}||^{1/m_9},||(A-I)^{m_9+1}||^{1/{(m_9+1)}}\}$ \label{Step_alpha}
    \State $finish=\alpha_{m_9}\le\Theta_{m_9}$
    \If {\textbf{not} $finish$}
        \State $A\leftarrow A_2$, storing $A$ to save one matrix product
        \State $A\leftarrow A^{1/2}$, using scaled DB iteration \eqref{ScalDB}. Note that $(A-I)^2=A_2-2A+I$ can be computed with no matrix product evaluations \label{Step_DB}
        \State $s=s+1$
    \EndIf
\EndWhile \label{Step_end_invscaling}
\State Find  the smallest integer $1\le k\le 9$ such that $\alpha_{m_k}\le\Theta_{m_k}$ \label{StepFinminm}
\State Compute $L = {-S^T_{m_{k}}}(-(A-I))$, using the expressions from Section \ref{Sec_polynomial} using $A-I$ and saving one matrix product using $(A-I)^2=A_2-2A+I$ if $s\ge 1$. \label{Step_Logapprox}
\State $L\leftarrow 2^sL$ \label{Step_squaring}
\State Postprocessing of matrix $L$: $L\leftarrow TLT^{-1}$\label{Step_Postproc}
\end{algorithmic}
\end{algorithm}

Step \ref{Step_Preproc} mirrors the preprocessing procedure outlined in \cite[Alg. 1]{Iba2021Improved}, utilizing the MATLAB function \texttt{balance}.
Steps \ref{Step_start_invscaling} to \ref{Step_end_invscaling} bring $A$ close to $I$ by taking successive square roots of $A$ until $\alpha_{m_9}\le\Theta_{m_9}$, where $m_9$ is the maximum Taylor order of the approximations $S_m^T$ and $\alpha_{m_9}$ is computed as described in Step \ref{Step_alpha}.

To compute the matrix square roots required, we have implemented a MATLAB function that uses the scaled Denman-Beavers iteration described in \cite[Eq. 6.28]{High08}. This iteration is defined by:
 {\setlength\arraycolsep{2pt}
 \begin{eqnarray}\label{ScalDB}
X_{k+1}&=&\frac{1}{2}(\mu_kX_k+\mu_k^{-1}Y_k^{-1}), \nonumber\\
Y_{k+1}&=&\frac{1}{2}(\mu_kY_k+\mu_k^{-1}X_k^{-1}),
 \end{eqnarray}}with $X_0=A$, $Y_0=I$ and $\mu_k=|\det(X_k)\det(Y_k)|^{-1/(2n)}$, being $n$ the dimension of matrix $A$.

In Step \ref{Step_alpha}, we use the following approximation for calculating $\alpha_{m_9}$ from \eqref{def_alpha}, similarly to \cite{RSID16,sastre2019fast}:
\begin{equation}\label{alphaapprox}
\alpha_{m_9}\approx\max\{||(A-I)^{m_9}||^{1/m_9}||,||(A-I)^{m_9+1}||^{1/{(m_9+1)}}\}.
\end{equation}

The implementation of the Algorithm \ref{Alg_log_Taylor} has several optimizations with respect to \cite[Alg. 1]{Iba2021Improved}. To estimate the 1-norm of $||(A-I)^{m_9}||$ and $||(A-I)^{m_9+1}||$, the 1-norm estimation algorithm of \cite{High88} can be employed. However, the number of power norm estimations with respect to the function \logmpol from \cite{Iba2021Improved} are reduced by taking into account the following statements:
 \begin{enumerate}
 \item If at least one square root is computed in Step \ref{Step_DB}, then the original matrix, denoted by $A_2$, is saved and the square root is computed as $A=A_2^{1/2}$ using the Denman-Beavers iteration given above. Then, the following matrix can be evaluated and stored with no extra matrix products:
 \begin{equation}\label{AI2}
 A\_I_2=(A-I)^2=A_2-2A+I.
 \end{equation}Note that in order to save memory, $A\_I_2$ and $A_2$ use only one variable in the implementation. Since the first term of \eqref{alphaapprox} can be bounded by $||(A-I)^{m_9}||=||(A-I)^{28}||\le ||A-I)^2||^{14}= ||A\_I_2||^{14}$ then one gets
 \begin{equation}\label{normAm9}
 ||(A-I)^{28}||^{1/28}\le (||A\_I_2||^{14})^{1/28}\le ||A\_I_2||^{1/2}.
 \end{equation}Then the estimation of $||(A-I)^{m_9}||$ can be avoided if
 \begin{equation}\label{alpha2term1}
 ||A\_I_2||^{1/2}\le \Theta_{m_9}.
 \end{equation}

 \item If condition \eqref{alpha2term1} does not hold, then the 1-norm estimation algorithm of \cite{High88} for estimating $||(A-I)^{m_9}||$ is used with matrix $A\_I_2$ instead of $A-I$. This  improvement saves processing time because the power steps involved in estimating the norm of the matrix power are fewer since $(A-I)^{m_9}=(A-I)^{28}=(A\_I_2)^{14}$ instead of the steps of estimating the norm of $(A-I)^{28}$ using $A-I$.

 \item The estimation of $||(A-I)^{m_9+1}||$ is also avoided if $||(A-I)^{m_9}||^{1/m_9}>\Theta_{m_9}$, or
 {\setlength\arraycolsep{2pt}
 \begin{eqnarray}\label{normAm9plus1}
 ||(A-I)^{m_9+1}||^{1/(m_9+1)}&\le&(||(A-I)^{28}||||(A-I)||)^{1/29}\nonumber\\
 &\le& (||A\_I_2||^{14}||A-I||)^{1/29}\le \Theta_{m_9},
 \end{eqnarray}}where the estimation of $||(A-I)^{28}||$ is used if available.
 \end{enumerate}

 When steps \ref{Step_start_invscaling}-\ref{Step_end_invscaling} are finished, it follows that $\alpha_{m_9}\le \Theta_9$, where $\alpha_{m_9}$ is given by \eqref{alphaapprox}. Then, Step \ref{StepFinminm} finds out the most appropriate approximation $1\le k\le 9$ by using a more efficient algorithm than the bisection method used in \cite[Alg. 1]{Iba2021Improved}. The new algorithm is based on the following steps:
 \begin{enumerate}
 \item If $s=0$ then $ A\_I_2=(A-I)^2$ is calculated, since it is used in all the matrix logarithm polynomial approximations.
 \item Find the minimum index $k\ge 1$ of $m_k$ to test, denoted by $min\_im$. Note that by \eqref{def_alpha} one gets
\begin{equation}\label{alphajm}
\alpha_{m_j}\ge \alpha_{m_9}\ge \max\{||(A-I)^{m_9}||^{1/m_9},||(A-I)^{m_9+1}||^{1/(m_9+1)}\}, \ 0\le j<9.
\end{equation}
 Therefore, if $||(A-I)^{m_9}||$ has been estimated,
 then
\begin{equation}\label{min_in}
 min\_im=\min\{j \ge 1, ||(A-I)^{m_9}||^{1/m_9}\le\Theta_j\}.
\end{equation}If $||(A-I)^{m_{9}+1}||$ has been also estimated then $min\_im$ is updated by
 \begin{equation}\label{min_inp}
 min\_im=\max\{min\_in,\min\{j \ge 1, ||(A-I)^{m_9+1}||^{1/(m_9+1)}\le\Theta_j\}\}.
 \end{equation}If no norm estimations have been done, then a rough approximation to $min\_im$ is the maximum value $j \ge 1$ such that
 \begin{equation}\label{roughmin_in}
min\_im=\min\{j \ge 1, ||A\_I_2||^{1/2}\le\Theta_j\}.
 \end{equation}For non-normal matrices $min\_im$ may be lower, then we reduce its value by taking $min\_im=tabmin\_im(min\_im)$ where $tabmin\_im$ is the following vector $tabmin\_im=\{1,\, 2,\, 3,\, 4,\, 4,\, 5,\, 5,\, 6,\, 6\}$. This ensures testing lower values of $min\_im$ for indices $k\geq 4$ of $m_k$. Moreover, if $||A\_I_2||^{1/2}>\Theta_{m_1}$ $min\_im$ is updated by $min\_im=\max\{min\_im,2\}$, else $min\_im=1$ is the first index to test.

 \item Find the maximum value of the index of $m_k$, $k\le 9$ to test, denoted by $max\_in$. For that, using \cite[Th. 1]{SIRD13} the following bound holds
{\setlength\arraycolsep{2pt}
\begin{eqnarray}\label{alpharoughbound_even}
\alpha_{m_j}&\le& \max\{(||A\_I_2||^{m_j/2})^{1/m_j}||,(||A\_I_2||^{m_j/2}||||A-I||)^{1/{(m_j+1)}}\}\nonumber\\
&\le& (||A\_I_2||^{m_j/2}||||A-I||)^{1/(m_j+1)},  \text{ for even } m_j,
\end{eqnarray}}where $A\_I_2$ is given by \eqref{AI2}, and similarly
{\setlength\arraycolsep{2pt}
\begin{eqnarray}\label{alpharoughbound_odd}
\alpha_{m_j}\le (||A\_I_2||^{(m_j-1)/2}||||A-I||)^{1/m_j}, \text{ for odd } m_j.
\end{eqnarray}}Then
{\setlength\arraycolsep{2pt}
\begin{eqnarray}\label{roughmax_even}
&&max\_in=\\
&&\min\{j,\, 1\le j \le 9,\text{ such that } (||A\_I_2||^{m_j/2}||||A-I||)^{1/(m_j+1)}\le \Theta_j\}\nonumber\\
&&  \text{for even } m_j,\nonumber
\end{eqnarray}}and
{\setlength\arraycolsep{2pt}
\begin{eqnarray}\label{roughmax_odd}
&&max\_in=\\
&&\min\{j,\, 1\le j \le 9, \text{ such that } (||A\_I_2||^{(m_j-1)/2}||||A-I||)^{1/m_j}\le \Theta_j\}\nonumber\\&&  \text{for odd } m_j.\nonumber
\end{eqnarray}}

 \item if $min\_in\ge max\_in$ then order $m_{max\_in}$ is selected. Else, we will start testing if $\alpha_{m_j}\le\Theta_{m_j}$ from $j=min\_in$, reducing the number of estimations of matrix power norms in a similar way as explained in the above enumeration.

 \item Moreover, if $\alpha_{m_j}\ge\Theta_{m_j}$ but $\alpha_{m_j}\le\Theta_{m_{j+1}}$ then $m_{j+1}$ is selected with no extra estimations of matrix power norms.
 \end{enumerate}

 In Step \ref{Step_Logapprox}, the matrix logarithm approximation of $A-I$ is computed by means of the efficient polynomial evaluations presented in Section \ref{Sec_polynomial} saving one matrix product if at least one square root has been computed and, then, variable $A\_I_2$ from \eqref{AI2} can be computed with no matrix product evaluations.

Finally, the squaring phase takes place in Step \ref{Step_squaring} and, in Step \ref{Step_Postproc}, the postprocessing operation $L=TLT^{-1}$ is evaluated to provide the logarithm of matrix A if the preprocessing step was applied.

%% file: Sec_tests.tex
\section{Numerical experiments} \label{Sec_experiments}
In this section, we compare the following MATLAB codes in terms of
accuracy and efficiency. The matrix test sets 1, 2 and 3 and the
simulation setup is similar to that on \cite[Sec.
3]{al2012improved}, and we briefly summarize the setup:
\textbf{\logmnew} follows \cite[Algorithm~4.1]{al2012improved},
\textbf{\texttt{logm}} is the implementation in MATLAB R2015b,
\textbf{\texttt{logm\_ex}} represents the Padé approach in
\cite{fasi2020dual}, \textbf{\logmpol} is
\cite[Algorithm~1]{Iba2021Improved} and \textbf{\logmpcg} represents
the new method proposed in this paper.

Table \ref{T_times} shows the average execution time in seconds to
carry out 11 executions, taking into account only the last 10
executions. \texttt{logm\_ex} was excluded from this comparison
since according to  the authors the implementation provided was not
optimized and its processing time was much higher than the
processing time from the other functions.
If the relative comparison in \% of the execution time of \logmpcg
with respect to the execution time of each of the codes is given by
\begin{equation}\label{timecomparison}
    \frac{t_{\texttt{code}}-t_{\texttt{\logmpcg}}}{t_{\logmpcg}}\times 100,
\end{equation}then, the comparison ranges between a 36,53 \% to 55,15
\% lower processing time with respect to the polynomial
approximation based function \logmpol. The processing time savings
are more noticeable when comparing to the Pad\'e based function
\logmnew, ranging between 114,14 \% to 334.63 \%. With respect to
the comparison with \texttt{logm}, \logmpcg is a MATLAB
implementation with no optimized built-in functions like
\texttt{logm}. Therefore, code \logmpcg should be improved as a
BLAS-like implementation to be able to beat \texttt{logm} in terms
of processing time.

Table \ref{T_matrix_products} shows the cost of functions \logmpcg
and \logmpol in terms of matrix products. If the relative cost
comparison between the two functions in \% is given by
\begin{equation}\label{Mcomparison}
\frac{\textrm{Cost}_{\logmpol}-\textrm{Cost}_{\logmpcg}}{\textrm{Cost}_{\logmpcg}}\times
100,
\end{equation}Table \ref{T_matrix_products} shows that the function \logmpcg is
between a 12.18~\% and a 13.33~\% more efficient than \logmpol. The
differences between the cost comparison in terms of matrix products
with respect to processing times from Table \ref{T_times} are due to
the more efficient scaling algorithm applied in \logmpcg. This
algorithm saves and accelerates estimations of matrix powers, and
also saves the computation of the matrix power $(A-I)^2$ needed for
evaluating the polynomial approximations if at least one matrix
square root is done. That is the case for all the matrices from
Tests 1 and 2 and for the majority of the matrices from Test~3.

\begin{table}[t]
    \begin{center}
        \begin{tabular}{| c || c | c | c| c| }
            \hline
              &\logmpcg&  \logmpol & \logmnew & \texttt{logm}\\ \hline
            Test 1 & 7.28&9.94& 15.59 & 3.84 \\ \hline
            Test 2 &7.08 &9.99 &16.12 & 3.92\\ \hline
            Test 3 &2.05&3.18&8.91&1.66\\ \hline
        \end{tabular}
        \caption{Execution times of codes in seconds.}
        \label{T_times}
    \end{center}
\end{table}
%

\begin{table}[t]
    \begin{center}
        \begin{tabular}{| c || c | c | c | }
            \hline
            &Cost \logmpcg& Cost \logmpol &Comparison (\%) \\ \hline
            Test 1 &7080 &7964 & 12.49 \%\\ \hline
            Test 2 &6824 &7734 & 13.33 \% \\\hline
            Test 3 &2355 &2642 & 12.18 \% \\\hline
        \end{tabular}
        \caption{Cost of  \logmpcg  and \logmpol in terms of matrix products, and relative comparison between the two functions \eqref{Mcomparison} in \%.}
        \label{T_matrix_products}
    \end{center}
\end{table}


The performance  profiles for rounding error confirm that \logmpcg
and \logmpol based on polynomial approximations were the most
accurate functions for the three matrix sets. In Tests 1 and 2, the
performance graphs of both codes are similar, but in Test 3, code
\logmpol presents better performance, followed by code \logmpcg.
However, Fig.~\ref{fig_set3}a shows that the relative errors
obtained by these two algorithms are of very similar order. The
Pad\'e based codes \logmnew,  \texttt{logm} and \texttt{logm\_ex}
had lower and very similar performance.

\begin{figure}[H]
    \begin{tabular}{cccccc}
        \includegraphics[width=6cm]{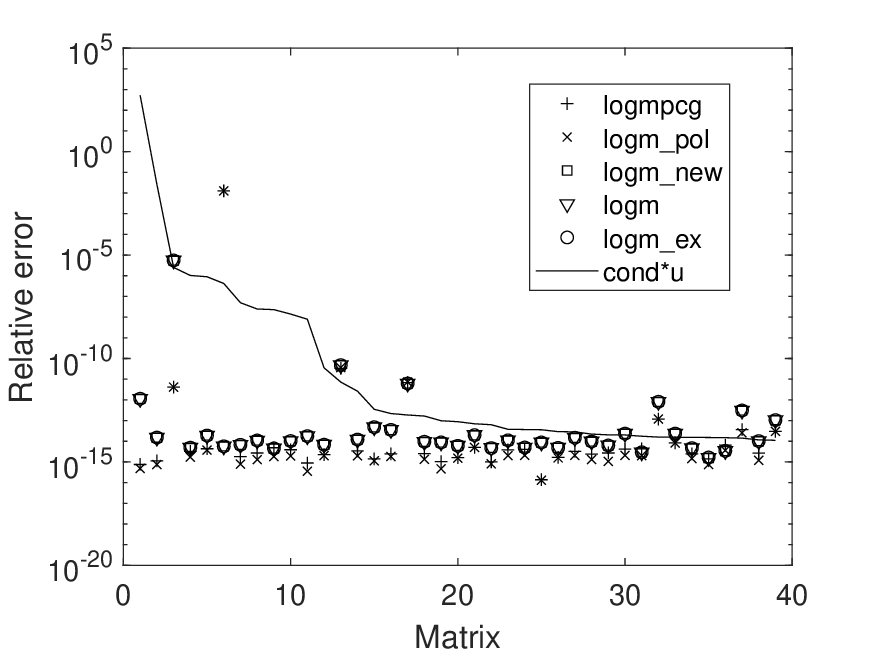}
        &\includegraphics[width=6cm]{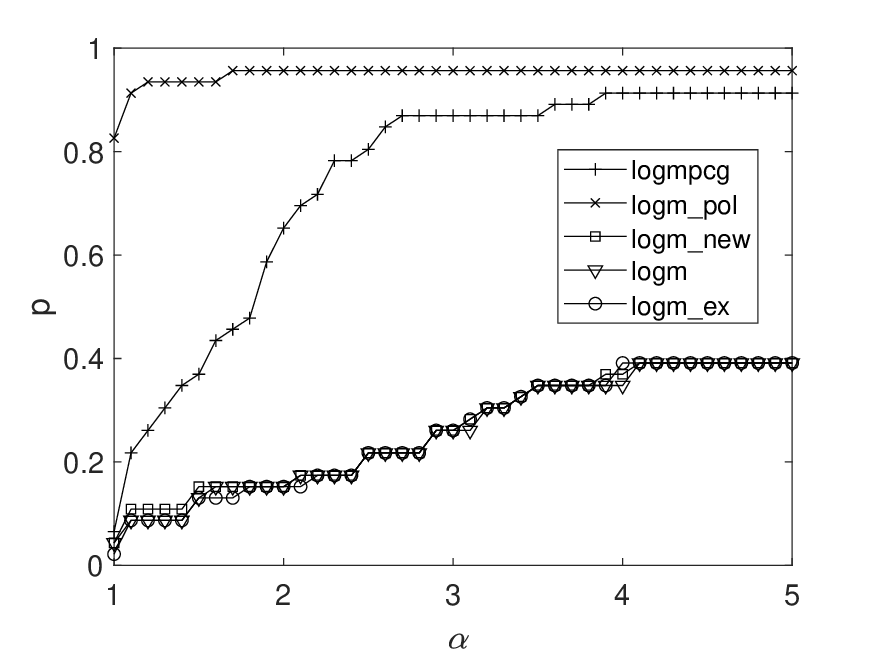}\\
        (\textbf{a}) Normwise relative errors. & (\textbf{b}) Performance profile.\\
    \end{tabular}
    \caption{Experimental rounding error results for test set~3.}
    \label{fig_set3}
\end{figure}

%% file: Logaritmo.bbl
\begin{thebibliography}{10}
\expandafter\ifx\csname url\endcsname\relax
  \def\url#1{\texttt{#1}}\fi
\expandafter\ifx\csname urlprefix\endcsname\relax\def\urlprefix{URL }\fi
\expandafter\ifx\csname href\endcsname\relax
  \def\href#1#2{#2} \def\path#1{#1}\fi

\bibitem{High08}
N.~J. Higham, Functions of Matrices: {Theory} and Computation, SIAM,
  Philadelphia, PA, USA, 2008.

\bibitem{miyajima2019verified}
S.~Miyajima, Verified computation for the matrix principal logarithm, Linear
  Algebra Appl. 569 (2019) 38--61.

\bibitem{hesselmann2010random}
A.~He{\ss}elmann, A.~G{\"o}rling, Random phase approximation correlation
  energies with exact {K}ohn--{S}ham exchange, Molecular Physics 108~(3-4)
  (2010) 359--372.

\bibitem{zachos2007classical}
C.~K. Zachos, A classical bound on quantum entropy, Journal of Physics A:
  Mathematical and Theoretical 40~(21) (2007) F407.

\bibitem{schenk2009modeling}
T.~Schenk, I.~Richardson, M.~Kraska, S.~Ohnimus, Modeling buckling distortion
  of dp600 overlap joints due to gas metal arc welding and the influence of the
  mesh density, Comput. Mat. Sci. 46~(4) (2009) 977--986.

\bibitem{horenko2008likelihood}
I.~Horenko, C.~Sch{\"u}tte, Likelihood-based estimation of multidimensional
  {L}angevin models and its application to biomolecular dynamics, Multiscale
  Modeling \& Simulation 7~(2) (2008) 731--773.

\bibitem{han2015large}
I.~Han, D.~Malioutov, J.~Shin, Large-scale log-determinant computation through
  stochastic {C}hebyshev expansions, in: Proc. 32nd Inter. Conf on Machine
  Learning, 6–11 July 2015, pp. 908--917.

\bibitem{fitzsimons2017entropic}
J.~Fitzsimons, D.~Granziol, K.~Cutajar, M.~Osborne, M.~Filippone, S.~Roberts,
  Entropic trace estimates for log determinants, in: Ceci M., Hollmén J.,
  Todorovski L., Vens C., Džeroski S. (eds) Machine Learning and Knowledge
  Discovery in Databases. ECML PKDD 2017. Lecture Notes in Computer Science,
  Vol. 10534, Springer, 2017, pp. 323--338.

\bibitem{williams1999matrix}
P.~M. Williams, Matrix logarithm parametrizations for neural network covariance
  models, Neural Networks 12~(2) (1999) 299--308.

\bibitem{grindrod2014dynamical}
P.~Grindrod, D.~J. Higham, A dynamical systems view of network centrality,
  Proceedings of the Royal Society A: Mathematical, Physical and Engineering
  Sciences 470~(2165) (2014) 20130835.

\bibitem{israel2001finding}
R.~B. Israel, J.~S. Rosenthal, J.~Z. Wei, Finding generators for {M}arkov
  chains via empirical transition matrices, with applications to credit
  ratings, Mathematical finance 11~(2) (2001) 245--265.

\bibitem{singer1976representation}
B.~Singer, S.~Spilerman, The representation of social processes by markov
  models, American Journal of Sociology 82~(1) (1976) 1--54.

\bibitem{ossikovski2015differential}
R.~Ossikovski, A.~De~Martino, Differential {M}ueller matrix of a depolarizing
  homogeneous medium and its relation to the {M}ueller matrix logarithm, JOSA A
  32~(2) (2015) 343--348.

\bibitem{ramezani2015non}
H.~Ram{\'e}zani, J.~Jeong, Non-linear elastic micro-dilatation theory: {M}atrix
  exponential function paradigm, International Journal of Solids and Structures
  67 (2015) 1--26.

\bibitem{rossignac2011steady}
J.~Rossignac, {\'A}.~Vinacua, Steady affine motions and morphs, ACM
  Transactions on Graphics (TOG) 30~(5) (2011) 1--16.

\bibitem{lastman1991infinite}
G.~Lastman, N.~Sinha, Infinite series for logarithm of matrix, applied to
  identification of linear continuous-time multivariable systems from
  discrete-time models, Electronics Letters 27~(16) (1991) 1468--1470.

\bibitem{crouch1999casteljau}
P.~Crouch, G.~Kun, F.~S. Leite, The {D}e {C}asteljau algorithm on lie groups
  and spheres, Journal of Dyn., Contr. Syst. 5~(3) (1999) 397--429.

\bibitem{KeLa89_2}
C.~Kenney, A.~J. Laub, Condition estimates for matrix functions, SIAM J. Matrix
  Anal. Appl. 10~(2) (1989) 191--209.

\bibitem{sastre2018efficient}
J.~Sastre, Efficient evaluation of matrix polynomials, Linear Algebra Appl. 539
  (2018) 229--250.

\bibitem{Jar2021}
E.~Jarlebring, M.~Fasi, E.~Ringh, Computational graphs for matrix functions,
  ACM Trans. Math. Softw. 48~(4) (2023) 1--35.

\bibitem{Iba2021Improved}
J.~Ibáñez, J.~Sastre, P.~Ruiz, J.~M. Alonso, E.~Defez, An improved taylor
  algorithm for computing the matrix logarithm, Mathematics 9~(17) (2021) 2018.

\bibitem{KeLa89_1}
C.~Kenney, A.~J. Laub, Padé error estimates for the logarithm of a matrix,
  Inter. Journal of Contr. 50~(3) (1989) 707--730.

\bibitem{dieci1996computational}
L.~Dieci, B.~Morini, A.~Papini, Computational techniques for real logarithms of
  matrices, SIAM J. Matrix Anal. Appl. 17~(3) (1996) 570--593.

\bibitem{High01}
N.~J. Higham, Evaluating {Pad\'e} approximants of the matrix logarithm, SIAM J.
  Matrix Anal. Appl. 22~(4) (2001) 1126--1135.

\bibitem{ChHK01}
S.~H. Cheng, N.~J. Higham, C.~S. Kenney, A.~J. Laub, Approximating the
  logarithm of a matrix to specified accuracy, SIAM J. Matrix Anal. Appl.
  22~(4) (2001) 1112--1125.

\bibitem{al2012improved}
A.~H. Al-Mohy, N.~J. Higham, Improved inverse scaling and squaring algorithms
  for the matrix logarithm, SIAM Journal on Scientific Computing 34~(4) (2012)
  C153--C169.

\bibitem{al2013computing}
A.~H. Al-Mohy, N.~J. Higham, S.~D. Relton, Computing the fréchet derivative of
  the matrix logarithm and estimating the condition number, SIAM Journal on
  Scientific Computing 35~(4) (2013) C394--C410.

\bibitem{fasi2018multiprecision}
M.~Fasi, N.~J. Higham, Multiprecision algorithms for computing the matrix
  logarithm, SIAM J. Matrix Anal. Appl. 39~(1) (2018) 472--491.

\bibitem{fasi2020dual}
M.~Fasi, B.~Iannazzo, The dual inverse scaling and squaring algorithm for the
  matrix logarithm, IMA J. Numer. Anal. 42~(3) (2022) 2829–2851.

\bibitem{kenney1998schur}
C.~S. Kenney, A.~J. Laub, A {S}chur--{F}r\'echet algorithm for computing the
  logarithm and exponential of a matrix, SIAM J. Matrix Anal. Appl. 19~(3)
  (1998) 640--663.

\bibitem{cardoso2016matrix}
J.~R. Cardoso, R.~Ralha, Matrix arithmetic-geometric mean and the computation
  of the logarithm, SIAM J. Matrix Anal. Appl. 37~(2) (2016) 719--743.

\bibitem{hale2008computing}
N.~Hale, N.~J. Higham, L.~N. Trefethen, Computing {A}$^{\alpha}$, $\log{(A)}$,
  and related matrix functions by contour integrals, SIAM J. Numer. Anal.
  46~(5) (2008) 2505--2523.

\bibitem{tatsuoka2020algorithms}
F.~Tatsuoka, T.~Sogabe, Y.~Miyatake, S.-L. Zhang, Algorithms for the
  computation of the matrix logarithm based on the double exponential formula,
  J. Comput. Appl. Math. 373 (2020) 112396.

\bibitem{BlDo99}
S.~Blackford, J.~Dongarra, {LAPACK} working note 41, installation guide for
  {LAPACK}, Tech. rep., Department of Computer Science University of Tennessee
  (1999).

\bibitem{SASTRE2018229}
J.~Sastre, Efficient evaluation of matrix polynomials, Linear Algebra Appl. 539
  (2018) 229--250.

\bibitem{PaSt73}
M.~S. Paterson, L.~J. Stockmeyer, On the number of nonscalar multiplications
  necessary to evaluate polynomials, SIAM J. Comp 2~(1) (1973) 60--66.

\bibitem{FASI2019182}
M.~Fasi, Optimality of the {Paterson–Stockmeyer} method for evaluating matrix
  polynomials and rational matrix functions, Linear Algebra Appl. 574 (2019)
  182--200.

\bibitem{SaIb21}
J.~Sastre, J.~Ib\'a\~nez, Evaluation of matrix polynomials beyond the
  {P}aterson-{S}tockmeyer method, Mathematics 9~(14) (2021) 1600.

\bibitem{sastre2019boosting}
J.~Sastre, J.~Ib{\'a}{\~n}ez, E.~Defez, Boosting the computation of the matrix
  exponential, Applied Mathematics and Computation 340 (2019) 206--220.

\bibitem{High02}
N.~J. Higham, Accuracy and Stability of Numerical Algorithms, SIAM,
  Philadelphia, PA, USA, 2002.

\bibitem{RSID16}
P.~Ruiz, J.~Sastre, J.~Ib{\'a}{\~n}ez, E.~Defez, High performance computing of
  the matrix exponential, Journal of Computational and Applied Mathematics 291
  (2016) 370--379.

\bibitem{sastre2019fast}
J.~Sastre, J.~Ib{\'a}{\~n}ez, P.~Alonso-Jord{\'a}, J.~Peinado, E.~Defez, Fast
  {T}aylor polynomial evaluation for the computation of the matrix cosine,
  Journal of Computational and Applied Mathematics 354 (2019) 641--650.

\bibitem{High88}
N.~J. Higham, Fortran codes for estimating the one-norm of a real or complex
  matrix, with applications to condition estimation, ACM Trans. Math. Softw.
  14~(4) (1988) 381--396.

\bibitem{SIRD13}
J.~Sastre, J.~{Ib\'a\~nez}, P.~Ruiz, E.~Defez, Efficient computation of the
  matrix cosine, Appl. Math. Comput. 219 (2013) 7575--7585.

\end{thebibliography}
